\documentclass[11pt,reqno]{amsart}
\usepackage{amsmath,amssymb}

 \makeatletter
 \evensidemargin  \oddsidemargin
 \makeatother

\begin{document}
\thispagestyle{empty}
 \setcounter{page}{1}
\begin{center}
{\Large\bf  Jordan $*$-homomorphisms on $C^*$-algebras }

\vskip.30in

{\small\bf  M. Eshaghi Gordji } \\[2mm]

{\footnotesize Department of Mathematics,
Semnan University,\\ P. O. Box 35195-363, Semnan, Iran\\
[-1mm] e-mail: {\tt madjid.eshaghi@gmail.com}}

{\small\bf  N. Ghobadipour } \\[2mm]

{\footnotesize Department of Mathematics,
Semnan University,\\ P. O. Box 35195-363, Semnan, Iran\\
[-1mm] e-mail: {\tt ghobadipour.n@gmail.com}}

{\bf  Choonkil Park} \\[2mm]

{\footnotesize Department of Mathematics,
Hanyang University,\\ Seoul 133-791, South Korea\\
[-1mm] e-mail: {\tt baak@hanyang.ac.kr}}
\end{center}
\vskip 5mm \noindent{\footnotesize{\bf Abstract.}
In this paper, we
investigate Jordan $*$-homomorphisms on $C^*$-algebras associated
with the following functional inequality
$\|f(\frac{b-a}{3})+f(\frac{a-3c}{3})+f(\frac{3a+3c-b}{3})\| \leq
\|f(a)\|.$ We moreover prove the superstability and the generalized
Hyers-Ulam stability of Jordan $*$-homomorphisms on
$C^*$-algebras associated with the following functional equation
$$f(\frac{b-a}{3})+f(\frac{a-3c}{3})+f(\frac{3a+3c-b}{3})=f(a).$$

 \vskip.10in

  \newtheorem{df}{Definition}[section]
  \newtheorem{rk}[df]{Remark}
   \newtheorem{lem}[df]{Lemma}
   \newtheorem{thm}[df]{Theorem}
   \newtheorem{pro}[df]{Proposition}
   \newtheorem{cor}[df]{Corollary}
   \newtheorem{ex}[df]{Example}

 \setcounter{section}{0}
 \numberwithin{equation}{section}

\vskip .2in

\begin{center}
\section{Introduction}
\end{center}

The stability of functional equations was first introduced
 by Ulam \cite{U} in 1940. More precisely, he proposed
the following problem: Given a group $G_1,$ a metric group $(G_2,d)$
and a positive number $\epsilon$, does there exist a $\delta>0$ such
that if a function $f:G_1\longrightarrow G_2$ satisfies the
inequality $d(f(xy),f(x)f(y))<\delta$ for all $x,y\in G1,$ then
there exists a homomorphism $T:G_1\to G_2$ such that $d(f(x),
T(x))<\epsilon$ for all $x\in G_1?$ As mentioned above, when this
problem has a solution, we say that the homomorphisms from $G_1$ to
$G_2$ are stable. In 1941, Hyers \cite{H} gave a partial solution
of Ulam's problem for the case of approximate additive mappings
under the assumption that $G_1$ and $G_2$ are Banach spaces. In
1978, Th. M. Rassias \cite{R2} generalized the theorem of Hyers by
considering the stability problem with unbounded Cauchy differences.
This phenomenon of stability that was introduced by Th. M. Rassias
\cite{R2} is called {\it generalized Hyers-Ulam stability} or {\it  Hyers-Ulam-Rassias stability}.

\begin{thm}\label{t1} Let $f:{E}\longrightarrow{E'}$ be a mapping from
 a norm vector space ${E}$
into a Banach space ${E'}$ subject to the inequality
$$\|f(x+y)-f(x)-f(y)\|\leq \epsilon (\|x\|^p+\|y\|^p) \eqno(1.1)$$
for all $x,y\in E,$ where $\epsilon$ and p are constants with
$\epsilon>0$ and $p<1.$ Then there exists a unique additive
mapping $T:{E}\longrightarrow{E'}$ such that
$$\|f(x)-T(x)\|\leq \frac{2\epsilon}{2-2^p}\|x\|^p \eqno(1.2)$$ for all $x\in E.$
If $p<0$ then inequality $(1.1)$ holds for all $x,y\neq 0$, and
$(1.2)$ for $x\neq 0.$ Also, if the function $t\mapsto f(tx)$ from
$\Bbb R$ into $E'$ is continuous for each fixed $x\in E,$ then T
is $\Bbb R$-linear.
\end{thm}

During the last decades several stability problems of functional
equations have been investigated by many mathematicians. A large
list of references concerning the stability of functional equations
can be found in \cite{B}--\cite{R1}.

\begin{df}
Let $A,B$ be  two $C^*$-algebras. A $\Bbb C$-linear mapping $f:A
\to B $ is called a Jordan $*$-homomorphism if $$\left\{%
\begin{array}{ll}
    f(a^2)=f(a)^2, \\
    f(a^*)=f(a)^* \\
\end{array}%
\right.    $$ for all $a \in A.$
\end{df}

C. Park \cite{P1} introduced and investigated Jordan $*$-derivations on
$C^*$- algebras associated with the following functional
inequality $$\|f(a)+f(b)+kf(c)\|\leq \|kf(\frac{a+b}{k}+c)\|$$ for
some integer $k$ greater than $1$ and proved the  generalized
Hyers-Ulam stability of Jordan $*$-derivations on
$C^*$-algebras associated with the following functional equation
$$f(\frac{a+b}{k}+c)=\frac{f(a)+f(b)}{k}+f(c)$$ for some integer $k$
greater than $1$ (see also \cite{pp, K, Mos}).

In this paper, we investigate  Jordan $*$-homomorphisms on
$C^*$-algebras associated with the following functional inequality
$$\|f(\frac{b-a}{3})+f(\frac{a-3c}{3})+f(\frac{3a+3c-b}{3})\| \leq
\|f(a)\|.$$ We moreover prove the  generalized Hyers-Ulam
stability of Jordan $*$-homomorphisms on $C^*$-algebras associated
with the following functional equation
$$f(\frac{b-a}{3})+f(\frac{a-3c}{3})+f(\frac{3a+3c-b}{3})=f(a).$$

\vskip 5mm
%================================================================
\section{Jordan $*$-homomorphisms}
%\setcounter{equation}{0}
%================================================================

In this section, we investigate Jordan $*$-homomorphisms on $C^*$-
algebras. Throughout this section, assume that $A,B$ are two
$C^*$-algebras.

\begin{lem}\label{t2}
Let $f:A \to B$ be a mapping such that
$$\|f(\frac{b-a}{3})+f(\frac{a-3c}{3})+f(\frac{3a+3c-b}{3})\|_B \leq \|f(a)\|_B \eqno(2.1)$$
for all $a,b,c \in A.$ Then $f$ is additive.
\end{lem}

\begin{proof}
Letting $a=b=c=0$ in $(2.1),$ we get $$\|3f(0)\|_B \leq
\|f(0)\|_B.$$So $f(0)=0.$ Letting $a=b=0$ in $(2.1),$ we get
$$\|f(-c)+f(c)\|_B \leq \|f(0)\|_B=0$$ for all $c \in A.$
Hence $f(-c)=-f(c)$ for all $c \in A.$ Letting $a=0$ and $b=6c$ in
$(2.1),$ we get $$\|f(2c)-2f(c)\|_B \leq \|f(0)\|_B=0$$ for all $c
\in A.$ Hence $$f(2c)=2f(c)$$ for all $c \in A.$ Letting $a=0$ and
$b=9c$ in $(2.1),$ we get
$$\|f(3c)-f(c)-2f(c)\|_B \leq \|f(0)\|_B=0$$ for all $c \in
A.$ Hence $$f(3c)=3f(c)$$ for all $c \in A.$ Letting $a=0$ in
$(2.1),$ we get
$$\|f(\frac{b}{3})+f(-c)+f(c-\frac{b}{3})\|_B \leq
\|f(0)\|_B=0$$ for all $a,b,c \in A.$ So
$$f(\frac{b}{3})+f(-c)+f(c-\frac{b}{3})=0\eqno(2.2)$$ for all $a,b,c \in
A.$ Let $t_1=c-\frac{b}{3}$ and $t_2=\frac{b}{3}$ in $(2.2).$ Then
$$f(t_2)-f(t_1+t_2)+f(t_1)=0$$ for all $t_1,t_2 \in A$. This means that  $f$
is additive.
\end{proof}

Now we prove the superstability problem for Jordan
$*$-homomorphisms as follows.

\begin{thm}\label{t2}
Let $p<1$ and $\theta$ be nonnegative real numbers, and let $f:A
\to B$ be a mapping such that
$$\|f(\frac{b-a}{3})+f(\frac{a-3\mu c}{3})+\mu f(\frac{3a+3c-b}{3})\|_B \leq \|f(a)\|_B,\eqno(2.2)$$
$$\|f(a^2)-f(a)^2\|_B \leq \theta \|a\|^{2p},\eqno(2.3)$$
$$\|f(a^*)-f(a)^*\|_B \leq \theta \|a^*\|^{p}\eqno(2.4)$$
for all $\mu \in \Bbb T^1:=\{\lambda \in \Bbb C~~;|\lambda|=1\}$
and all $a,b,c \in A.$ Then the mapping $f:A \to B$ is a  Jordan
$*$-homomorphism.
\end{thm}

\begin{proof}
 Let $\mu=1$ in (2.2). By Lemma 2.1, the mapping $f:A \to B$ is
additive. Letting $a=b=0$ in $(2.2),$ we get
$$\|f(-\mu c)+\mu f(c)\|_B \leq \|f(0)\|_B=0$$ for all $c \in A$
and all $\mu \in \Bbb T^1.$ So $$-f(\mu c)+\mu f(c)=f(-\mu c)+\mu
f(c)=0$$ for all $c \in A$ and all $\mu \in \Bbb T^1.$ Hence $f(\mu
c)=\mu f(c)$ for all $c \in A$ and all $\mu \in \Bbb T^1.$ By
Theorem $2.1$ of \cite{P}, the mapping $f:A \to B$ is $\Bbb C$-linear.
It follows from $(2.3)$ that
\begin{align*}
\|f(a^2)-f(a)^2\|_B&=\|\frac{1}{n^2}f(n^2a^2)-(\frac{1}{n}f(na))^2\|_B\\
&=\frac{1}{n^2}
\|f(n^2a^2)-f(na)^2\|_B\\
&\leq \frac{\theta}{n^2} {n^{2p}}\|a\|^{2p}
\end{align*}
for all $a \in A.$ Thus, since $p<1,$ by letting n tend to
$\infty$ in last inequality, we obtain $f(a^2)=f(a)^2$ for all $a
\in A.$ On the other hand, it follows from $(2.4)$ that
\begin{align*}
\|f(a^*)-f(a)^*\|_B&=\|\frac{1}{n}f(na^*)-(\frac{1}{n}f(na))^*\|_B\\
&=\frac{1}{n}
\|f(na^*)-f(na)^*\|_B\\
&\leq \frac{\theta}{n} {n^{p}}\|a^*\|^{p}
\end{align*}
 for all $a \in A.$ Thus, since $p<1,$ by letting n tend to
$\infty$ in last inequality, we obtain $f(a^*)=f(a)^*$ for all $a
\in A.$ Hence the mapping $f:A \to B$ is a  Jordan
$*$-homomorphism.
\end{proof}
\begin{thm}\label{t2}Let $p>1$ and $\theta$ be a nonnegative real
number, and let $f:A \to B$ be a mapping satisfying $(2.2),$
$(2.3)$ and $(2.4).$ Then the mapping $f:A \to B$ is a Jordan
$*$-homomorphism.
\end{thm}

\begin{proof}
The proof is similar to the proof of Theorem 2.2.
\end{proof}

We prove the generalized Hyers-Ulam stability of Jordan
$*$-homomorphisms on $C^*$-algebras.

\begin{thm}\label{t2}
Suppose that  $f:A \to B$ is an odd  mapping  for which there exists
a function $\varphi: A \times A \times A \to \Bbb R^+$ such that
$$\sum_{i=0}^{\infty}3^i \varphi(\frac{a}{3^i},\frac{b}{3^i},
\frac{c}{3^i})< \infty, \eqno (2.5)$$
$$\lim_{n \to \infty}3^{2n}\varphi(\frac{a}{3^n},\frac{b}{3^n},\frac{c}{3^n})=0, \eqno (2.6)$$
$$\|f(a^*)-f(a)^*\|_B \leq \varphi(a,a,a), \eqno (2.7)$$
$$\|f(\frac{\mu b-a}{3})+f(\frac{a-3c}{3})+ \mu f(\frac{3a-b}{3}+c)-f(a)+f(c^2)-f(c)^2\|_B \leq \varphi(a,b,c) \eqno (2.8)$$
for all $a,b,c \in A$ and all $\mu \in \Bbb T^1.$ Then there
exists a unique  Jordan $*$-homomorphism $h:A \to B$ such that
$$\|h(a)-f(a)\|_B \leq \sum_{i=0}^{\infty}3^i
\varphi(\frac{a}{3^i},\frac{2a}{3^i},0) \eqno (2.9)$$ for all $a \in
A.$
\end{thm}

\begin{proof}
Letting $\mu=1,$~ $b=2a$ and $c=0$ in $(2.8),$ we get
$$\|3f(\frac{a}{3})-f(a)\|_B \leq \varphi(a,2a,0)$$ for all $a \in
A.$  Using the induction method, we have $$$$
$$\|3^nf(\frac{a}{3^n})-f(a)\|\leq \sum_{i=0}^{n-1}3^i
\varphi(\frac{a}{3^i},\frac{2a}{3^i},0) \eqno (2.10)$$ for all $a
\in A.$ In order to show the functions $h_n(a)=3^n f(\frac{a}{3^n})$
form a convergent sequence, we use the Cauchy convergence criterion.
Indeed, replace $a$ by $ \frac{a}{3^m} $ and multiply by $3^m$ in
$(2.10),$ where $m$ is an arbitrary positive integer. We find that
$$\|3^{m+n}f(\frac{a}{3^{m+n}})-3^m f(\frac{a}{3^m})\|\leq \sum_{i=m}^{m+n-1}3^i
\varphi(\frac{a}{3^i},\frac{2a}{3^i},0) \eqno (2.11)$$  for all
positive integers. Hence by the Cauchy criterion the limit
$h(a)=\lim_{n \to \infty}h_n(a)$ exists for each $a \in A.$ By
taking the limit as $n \to \infty$ in $(2.10)$ we see that
$$\|h(a)-f(a)\| \leq \sum_{i=0}^{\infty}3^i
\varphi(\frac{a}{3^i},\frac{2a}{3^i},0)$$ and $(2.9)$ holds for all
$a \in A.$ Let $\mu=1$ and $c=0$ in $(2.8),$ we get
$$\|f(\frac{b-a}{3})+f(\frac{a}{3})+f(\frac{3a-b}{3})-f(a)\|_B \leq \varphi(a,b,0)
\eqno(2.12)$$ for all $a,b,c \in A.$  Multiplying  both sides
$(2.12)$ by $3^n$ and Replacing $a,b$ by
$\frac{a}{3^n},\frac{b}{3^n},$  respectively, we get
$$\|3^nf(\frac{b-a}{3^{n+1}})+3^nf(\frac{a}{3^{n+1}})+3^nf(\frac{3a-b}{3^{n+1}})-3^nf(\frac{a}{3^n})\|_B \leq 3^n\varphi(\frac{a}{3^n},\frac{b}{3^n},0)
\eqno(2.13)$$ for all $a,b,c \in A.$ Taking the limit as $n \to
\infty,$ we obtain
$$h(\frac{b-a}{3})+h(\frac{a}{3})+h(\frac{3a-b}{3})-h(a)=0 \eqno(2.14)$$  for all $a,b,c \in A.$
Putting $b=2a$ in $(2.14),$ we get $3h(\frac{a}{3})=h(a)$ for all $a
\in A.$ Replacing $a$ by $2a$ in $(2.14),$ we get
$$h(b-2a)+h(6a-b)=2h(2a) \eqno(2.15)$$ for all $a,b \in A.$
Letting $b=2a$ in $(2.15),$ we get $h(4a)=2h(2a)$ for all $a \in A.$
So $h(2a)=2h(a)$ for all $a \in A.$ Letting $3a-b=s$ and $b-a=t$ in
$(2.14),$ we get
$$h(\frac{t}{3})+h(\frac{s+t}{6})+h(\frac{t}{3})=h(\frac{s+t}{2})$$ for all $s,t\in
A.$ Hence $h(s)+h(t)=h(s+t)$  for all $s,t\in A.$ So, $h$ is
additive. Letting $a=c=0$ in $(2.12)$ and using the above method, we
have $h(\mu b)=\mu h(b)$ for all $b \in A$ and all $\mu \in \Bbb T.$
Hence by the Theorem 2.1 of \cite{P}, the mapping $f:A \to B$
is $\Bbb C$-linear.\\
Now, let $h^{'}:A \to B$ be another $\Bbb C$-linear mapping
satisfying $(2.9).$ Then we have
\begin{align*}
\|h(a)-h^{'}(a)\|_B &=3^n
\|h(\frac{a}{3^n})-h^{'}(\frac{a}{3^n})\|_B \\
& \leq
3^n[\|h(\frac{a}{3^n})-f(\frac{a}{3^n})\|_B\\
&+\|h^{'}(\frac{a}{3^n})-f(\frac{a}{3^n})\|_B]\\
&\leq 2\sum_{i=n}^{\infty}3^i
\varphi(\frac{a}{3^i},\frac{2a}{3^i},0)=0
\end{align*}
for all $a \in A.$ Letting $\mu=1$ and $a=b=0$ in $(2.8),$ we get
$\|f(c^2)-f(c)^2\|_B \leq \varphi(0,0,c)$ for all $c \in A.$ So
$$\|h(c^2)-h(c)^2\|_B=\lim_{n \to \infty}3^{2n}\|f(\frac{c^2}{3^{2n}})-f(\frac{c}{3^n})^2\|_B \leq \lim_{n \to \infty}3^{2n}\varphi(0,0,\frac{c}{3^n})=0$$
for all $c \in A.$ Hence $h(c^2)=h(c)^2$ for all $c \in A.$ On the
other hand we have $$\|h(c^*)-h(c)^*\|_B=\lim_{n \to \infty}3^n
\|f(\frac{c^*}{3^n})-f(\frac{c}{3^n})^*\|_B \leq \lim_{n \to
\infty}3^n \varphi(\frac{c}{3^n},\frac{c}{3^n},\frac{c}{3^n})=0$$
for all $c \in A.$ Hence $h(c^*)=h(c)^*$ for all $c \in A.$ Hence
$h:A \to B$ is a unique Jordan $*$-homomorphism.
\end{proof}

\begin{cor}
Suppose that $f:A \to B$ is a mapping with $f(0)=0$ for which
there exists constant $\theta\geq0$ and $p_1,p_2,p_3>1$ such that
$$\|f(\frac{\mu b-a}{3})+f(\frac{a-3c}{3})+ \mu f(\frac{3a-b}{3}+c)-f(a)+f(c^2)-f(c)^2\|_B \leq \theta (\|a\|^{p_1}+\|b\|^{p_2}+\|c\|^{p_3}),$$
$$\|f(a^*)-f(a)^*\|_B \leq \theta (\|a\|^{p_1}+\|a\|^{p_2}+\|a\|^{p_3})$$
for all $a,b,c \in A$ and all $ \mu \in \Bbb T.$ Then there exists
a unique Jordan $*$-homomorphism $h:A \to B$ such that
$$\|f(a)-h(a)\|_B \leq \frac{\theta \|a\|^{p_1}}{1-3^{(1-p_1)}}+\frac{\theta 2^{p_2}\|a\|^{p_2}}{1-3^{(1-p_2)}}$$
for all $a \in A.$
\end{cor}

\begin{proof}
Letting $\varphi(a,b,c):=\theta
(\|a\|^{p_1}+\|b\|^{p_2}+\|c\|^{p_3})$ in Theorem 2.4, we obtain
the result.
\end{proof}
The following corollary is the Isac-Rassias stability.\\
\begin{cor}
Let $\psi: \Bbb R^+\cup\{0\} \to \Bbb R^+\cup\{0\}$ be a function
with $\psi(0)=0$ such that
$$\lim _{t \to 0}\frac{\psi(t)}{t}=0,$$
$$\psi(st)\leq \psi(s)\psi(t)~~~~~~~~~~s,t \in \Bbb R^+,$$
$$3\psi(\frac{1}{3})<1.$$\\
Suppose that $f:A \to B$ is a mapping with $f(0)=0$  satisfying
$(2.7)$ and $(2.8)$ such that
$$\|f(\frac{\mu b-a}{3})+f(\frac{a-3c}{3})+ \mu f(\frac{3a-b}{3}+c)-f(a)+f(c^2)-f(c)^2\|_B \leq
\theta [\psi(\|a\|)+\psi(\|b\|)+\psi(\|c\|)]$$ for all $a,b,c \in
A$ where $\theta>0$ is a constant. Then  there exists a unique
Jordan $*$-homomorphism $h:A \to B$ such that
$$\|h(a)-f(a)\|_B \leq \frac{\theta(1+\psi(2))\psi(\|a\|)}{1-3\psi(\frac{1}{3})}$$ for all $a \in A.$
\end{cor}
\begin{proof}
Letting $\varphi(a,b,c):=\theta
[\psi(\|a\|)+\psi(\|b\|)+\psi(\|c\|)]$ in Theorem 2.4, we obtain
the result.
\end{proof}

\begin{thm}
Suppose that $f:A \to B$ is a mapping with $f(0)=0$ for which there
exists a function $\varphi:A \times A \times A \to B$ satisfying
$(2.7),$~ $(2.8)$ and $(2.8)$ such that
$$\sum_{i=1}^{\infty}3^{-i}\varphi(3^ia,3^ib,3^ic)< \infty,  \eqno(2.16)$$
$$\lim_{n \to \infty}3^{-2n} \varphi(3^ia,3^ib,3^ic)=0  \eqno(2.17)$$
for all $a,b,c \in A.$  Then there exists a unique  Jordan
$*$-homomorphism $h:A \to B$ such that
$$\|h(a)-f(a)\|_B \leq \sum_{i=1}^{\infty}3^{-i}
\varphi(3^ia,3^i2a,0)  \eqno(2.18)$$ for all $a \in A.$
\end{thm}

\begin{proof}
Letting $\mu=1,$~ $b=2a$ and $c=0$ in $(2.8),$ we get
$$\|3f(\frac{a}{3})-f(a)\|_B \leq \varphi(a,2a,0) \eqno(2.19)$$ for all $a \in
A.$ Replacing $a$ by $3a$ in $(2.19),$ we get
$$\|3^{-1}f(3a)-f(a)\|_B \leq 3^{-1}\varphi(3a,2(3a),0)$$ for all
$a \in A.$ On can apply the induction method to prove that
$$\|3^{-n}f(3^na)-f(a)\|_B \leq
\sum_{i=1}^{n}3^{-i}\varphi(3^ia,2(3^ia),0)  \eqno(2.20)$$ for all
$a \in A.$
 In order to show the functions $h_n(a)=3^{-n}
f(3^na)$ form a convergent sequence, we use the Cauchy convergence
criterion. Indeed, replace $a$ by $3^ma $ and multiply by $3^{-m}$
in $(2.20),$ where $m$ is an arbitrary positive integer. We find
that
$$\|3^{-(m+n)}f(3^{m+n}a)-3^{-m} f(3^ma)\|\leq
\sum_{i=m+1}^{m+n}3^{-i} \varphi(3^ia,2(3^ia),0) \eqno (2.21)$$ for
all positive integers. Hence by the Cauchy criterion the limit
$h(a)=\lim_{n \to \infty}h_n(a)$ exists for each $a \in A.$ By
taking the limit as $n \to \infty$ in $(2.20)$ we see that
$$\|h(a)-f(a)\| \leq \sum_{i=1}^{\infty}3^{-i}
\varphi(3^ia,2(3^ia),0)$$ and $(2.18)$ holds for all $a \in A.$

The rest of the proof is similar to the proof of Theorem 2.4.
\end{proof}

\begin{cor}
Suppose that $f:A \to B$ is a mapping with $f(0)=0$ for which
there exists constant $\theta\geq0$ and $p_1,p_2,p_3<1$ such that
$$\|f(\frac{\mu b-a}{3})+f(\frac{a-3c}{3})+ \mu f(\frac{3a-b}{3}+c)-f(a)+f(c^2)-f(c)^2\|_B \leq \theta (\|a\|^{p_1}+\|b\|^{p_2}+\|c\|^{p_3}),$$
$$\|f(a^*)-f(a)^*\|_B \leq \theta (\|a\|^{p_1}+\|a\|^{p_2}+\|a\|^{p_3})$$
for all $a,b,c \in A$ and all $ \mu \in \Bbb T.$ Then there exists
a unique Jordan $*$-homomorphism $h:A \to B$ such that
$$\|f(a)-h(a)\|_B \leq \frac{\theta \|a\|^{p_1}}{3^{(1-p_1)}-1}+\frac{\theta 2^{p_2}\|a\|^{p_2}}{3^{(1-p_2)}-1}$$
for all $a \in A.$
\end{cor}

\begin{proof}
Letting $\varphi(a,b,c):=\theta
(\|a\|^{p_1}+\|b\|^{p_2}+\|c\|^{p_3})$ in Theorem 2.7, we obtain
the result.
\end{proof}

The following corollary is the Isac-Rassias stability.

\begin{cor}
Let $\psi: \Bbb R^+\cup\{0\} \to \Bbb R^+\cup\{0\}$ be a function
with $\psi(0)=0$ such that
$$\lim _{t \to 0}\frac{\psi(t)}{t}=0,$$
$$\psi(st)\leq \psi(s)\psi(t)~~~~~~~~~~s,t \in \Bbb R^+,$$
$$3^{-1}\psi(3)<1.$$\\
Suppose that $f:A \to B$ is a mapping with $f(0)=0$  satisfying
$(2.7)$ and $(2.8)$ such that
$$\|f(\frac{\mu b-a}{3})+f(\frac{a-3c}{3})+ \mu f(\frac{3a-b}{3}+c)-f(a)+f(c^2)-f(c)^2\|_B \leq
\theta [\psi(\|a\|)+\psi(\|b\|)+\psi(\|c\|)]$$ for all $a,b,c \in
A$ where $\theta>0$ is a constant. Then  there exists a unique
Jordan $*$-homomorphism $h:A \to B$ such that
$$\|h(a)-f(a)\|_B \leq \frac{\theta(1+\psi(2))\psi(\|a\|)}{1-3^{-1}\psi(3)}$$ for all $a \in A.$
\end{cor}

\begin{proof}
Letting $\varphi(a,b,c):=\theta
[\psi(\|a\|)+\psi(\|b\|)+\psi(\|c\|)]$ in Theorem 2.7, we obtain
the result.
\end{proof}

One can get easily the stability of
Hyers-Ulam  by the following corollary.

 \begin{cor}
 Suppose that $f:A \to B$ is a mapping
with $f(0)=0$ for which there exists constant $\theta\geq0$ such
that
$$\|f(\frac{\mu b-a}{3})+f(\frac{a-3c}{3})+ \mu f(\frac{3a-b}{3}+c)-f(a)+f(c^2)-f(c)^2\|_B \leq \theta ,$$
$$\|f(a^*)-f(a)^*\|_B \leq \theta $$
for all $a,b,c \in A$ and all $ \mu \in \Bbb T.$ Then there exists
a unique Jordan $*$-homomorphism $h:A \to B$ such that
$$\|f(a)-h(a)\|_B \leq \theta $$
for all $a \in A.$
\end{cor}

\begin{proof}
Letting $p_1=p_2=p_3=0$ in Corollary 2.8, we obtain the result.
\end{proof}

{\small
%----------------------------------------------------------------------%

}
\end{document}